\RequirePackage{fix-cm}
\documentclass[smallextended,envcountsame,natbib]{svjour3}       
\smartqed  
\usepackage{graphicx}
\usepackage{newtxtext}
\usepackage{newtxmath}

\usepackage{hyperref}

\usepackage{amsmath,amstext,amssymb
}
\usepackage{stmaryrd}
\usepackage{bm}

\usepackage[mathcal]{euscript}

\usepackage{proof}
\newcommand{\kleene}{^{\textstyle *}}
\renewcommand{\epsilon}{\varepsilon}

\bmdefine{\seqarrow}{\rightarrow}
\newcommand{\bs}{\backslash}
\newcommand{\varbox}{\Box^{\downarrow}}
\newcommand{\Typ}{\operatorname{Tp}}
\newcommand{\Pri}{\mathrm{Pr}}
\newcommand{\concat}{\:}
\newcommand{\contexthole}{\scalebox{0.7}{\ensuremath{\blacksquare}}}
\newcommand{\den}[1]{\llbracket #1 \rrbracket}

\newcommand{\selected}[1]{\,\fbox{\ensuremath{#1}}\,}

\newcommand{\qedhere}{\tag*{\qed}}

\allowdisplaybreaks

\title{On the Recognizing Power of the Lambek Calculus with Brackets}
\author{Makoto Kanazawa}

\institute{M. Kanazawa \at Department of Advanced Sciences, Faculty of Science and Engineering, Hosei University, 3--7--2 Kajino-cho, Koganei-shi, Tokyo 184--8584, Japan\\
\email{kanazawa@hosei.ac.jp}}

\begin{document}
\maketitle

\begin{abstract}
Every language recognized by the Lambek calculus with brackets is context-free.
This is shown by combining an observation by J\"ager with an entirely straightforward adaptation of the method Pentus used for the original Lambek calculus.
The case of the variant of the calculus allowing sequents with empty antecedents is slightly more complicated, requiring a restricted use of the multiplicative unit.
\keywords{Lambek calculus with brackets \and recognizing power}
\end{abstract}

\section{Introduction}

The calculus $\mathbf{L}\Diamond$, an enrichment of the Lambek calculus with brackets and associated residuation modalities, was introduced by \cite{Moortgat:1996}.
It is a kind of controlled mixture of the original Lambek calculus $\mathbf{L}$ \citep{Lambek:1958} and its nonassociative variant $\mathbf{NL}$ \citep{Lambek:1961}.
The question of its recognizing power was studied by \cite{Jaeger:2003}.
In terms of a natural definition of recognition he called ``$t$-recognition'', \cite{Jaeger:2003} put forward a proof that $\mathbf{L}\Diamond$ recognizes only context-free languages.
As pointed out by \cite{Kanovich-et-al:2017}, however, J\"ager's proof was flawed since it rested on the assumption that Versmissen's \citeyearpar{Versmissen:1996} translation from types of $\mathbf{L}\Diamond$ into types of $\mathbf{L}$ was a faithful embedding, which \cite{Fadda-and-Morrill:2005} showed not to be the case.
This paper provides a correct proof of context-freeness of $\mathbf{L}\Diamond$ as well as of the variant $\mathbf{L}\kleene\Diamond$ allowing empty antecedents.

\section{The Calculus $\mathbf{L}\Diamond$}

Let $\Pri = \{ p_1,p_2,p_3,\dots \}$ be an infinite supply of \emph{primitive types}.
If $\mathcal{B}$ is some set, we let $\Typ(\mathcal{B})$ denote the smallest superset of $\mathcal{B}$ such that $A,B \in \Typ(\mathcal{B})$ implies $A \bs B, B / A, A \bullet B, \Diamond A, \varbox A \in \Typ(\mathcal{B})$.
An element of $\Typ(\Pri)$ is called a \emph{type}.
We let upper-case letters $A,B,C,\dots$ range over types.
The \emph{length} $||A||$ of a type $A$ is defined by
\begin{gather*}
||p|| = 1\quad\text{if $p$ is a primitive type,}\\
||A \bs B|| = ||B / A|| = ||A \bullet B|| = ||A||+||B||,\\
||\Diamond A|| = ||\varbox A|| = ||A|| + 2.
\end{gather*}

A \emph{type tree} is either a single node labeled by a type or a tree with an unlabeled root all of whose immediate subtrees are type trees.
A \emph{type hedge} is a finite sequence of type trees, which is written without commas between trees.
Following \cite{Jaeger:2003}, we use angle brackets $\langle, \rangle$ to denote type trees and type hedges.
A simultaneous inductive definition of type trees and type hedges go as follows:
\begin{itemize}
\item
If $A$ is a type, then $A$ is a type tree.
\item
If $\Gamma$ is a type hedge, then $\langle \Gamma \rangle$ is a type tree.
\item
If $T_1,\dots,T_n$ ($n \geq 0$) are type trees, then $T_1 \dots T_n$ is a type hedge.
\end{itemize}
When $n=0$ in the last clause, the type hedge $T_1 \dots T_n$ is called \emph{empty}.
Note that every type tree is a type hedge.
We use upper-case Greek letters $\Pi, \Gamma, \Delta,\dots$ to denote type hedges.
If $\Pi$ and $\Gamma$ are type hedges, then $\Pi \concat \Gamma$ denotes the type hedge that is their concatenation.
The \emph{yield} of a type hedge $\Gamma$ is the string of types that label the leaves of $\Gamma$---in other words, the yield of $\Gamma$ is the result of removing all angle brackets from $\Gamma$.

A \emph{sequent} is an expression of the form
\[
\Gamma \seqarrow A
\]
where $\Gamma$ is a type hedge and $A$ is a type; $\Gamma$ is its \emph{antecedent} and $A$ its \emph{succedent}.

A \emph{context} is just like a type hedge, except that a special symbol $\contexthole$ labels exactly one leaf; all other labels are types.
A context is denoted by $\Pi[\contexthole], \Gamma[\contexthole], \Delta[\contexthole]$, etc.
If $\Gamma[\contexthole]$ is a context and $\Delta$ is a type hedge, then $\Gamma[\Delta]$ denotes the type hedge which is the result of replacing the unique leaf labeled by $\contexthole$ in $\Gamma$ by the hedge $\Delta$; in $\Gamma[\Delta]$, the siblings of $\contexthole$ in $\Gamma[\contexthole]$ become the siblings of the roots of the trees that make up $\Delta$.
A precise inductive definition goes as follows:
\begin{itemize}
\item
If $\Gamma[\contexthole]$ is a single node labeled by $\contexthole$, then $\Gamma[\Delta] = \Delta$.
\item
If $\Gamma[\contexthole] = \Pi_1 \concat T[\contexthole] \concat \Pi_2$, then $\Gamma[\Delta] = \Pi_1 \concat T[\Delta] \concat \Pi_2$.
\item
If $\Gamma[\contexthole] = \langle \Pi[\contexthole] \rangle$, then $\Gamma[\Delta] = \langle \Pi[\Delta] \rangle$.
\end{itemize}

The sequent calculus $\mathbf{L}\Diamond$ has the following rules of inference:
\begin{gather*}
\infer[(\bs{\seqarrow})]{\Delta[\Gamma \concat A \bs B] \seqarrow C}{
	\Gamma \seqarrow A
	&
	\Delta[B] \seqarrow C
}
\qquad
\infer[({\seqarrow}\bs)]{\Pi \seqarrow A \bs B}{
A \concat \Pi \seqarrow B
}
\\
\infer[(/{\seqarrow})]{\Delta[B/A \concat \Gamma] \seqarrow C}{
	\Gamma \seqarrow A
	&
	\Delta[B] \seqarrow C
}
\qquad
\infer[({\seqarrow}/)]{\Pi \seqarrow B/A}{
\Pi \concat A \seqarrow B
}
\\
\infer[(\bullet{\seqarrow})]{\Gamma[A \bullet B] \seqarrow C}{
	\Gamma[A \concat B] \seqarrow C
}
\qquad
\infer[({\seqarrow}\bullet)]{\Gamma \concat \Delta \seqarrow A \bullet B}{
	\Gamma \seqarrow A
	&
	\Delta \seqarrow B
}
\\
\infer[(\Diamond{\seqarrow})]{\Gamma[\Diamond A] \seqarrow B}{
	\Gamma[\langle A \rangle] \seqarrow B
}
\qquad
\infer[({\seqarrow}\Diamond)]{\langle \Gamma \rangle \seqarrow \Diamond A}{
	\Gamma \seqarrow A
}
\\
\infer[(\varbox{\seqarrow})]{\Gamma[\langle \varbox A \rangle] \seqarrow B}{
	\Gamma[A] \seqarrow B
}
\qquad
\infer[({\seqarrow}\varbox)]{\Gamma \seqarrow \varbox A}{
	\langle \Gamma \rangle \seqarrow A
}
\\
\infer[\text{Cut}]{\Delta[\Gamma] \seqarrow B}{
	\Gamma \seqarrow A
	&
	\Delta[A] \seqarrow B
}
\end{gather*}
In $({\seqarrow}\bs)$ and $({\seqarrow}/)$, the hedge $\Pi$ should not be empty.
An \emph{initial sequent} is a sequent of the form $p_i \seqarrow p_i$.%
\footnote{Equivalently, we may take all sequents of the form $A \seqarrow A$ as initial sequents, as \cite{Jaeger:2003} did.}
A sequent is \emph{provable} if it can be derived from initial sequents using rules of inference.
We write $\vdash_{\mathbf{L}\Diamond} \Gamma \seqarrow C$ if $\Gamma \seqarrow C$ is provable in $\mathbf{L}\Diamond$.
The cut rule is eliminable \citep{Moortgat:1996}, so every provable sequent has a cut-free proof.

Since the type hedge $\Pi$ is required to be nonempty in the rules $({\seqarrow}\bs)$ and $({\seqarrow}/)$ of $\mathbf{L}\Diamond$, the antecedent of a provable sequent is never empty, and $\langle \rangle$ (a matching pair of angle brackets with nothing in between) cannot appear in the antecedent of a provable sequent.
As in the case of the original Lambek calculus, the calculus without this restriction, referred to as $\mathbf{L}\kleene\Diamond$, may also be of interest.
We will discuss $\mathbf{L}\kleene\Diamond$ in Section~\ref{sec:L*b}.

An \emph{$\mathbf{L}\Diamond$ grammar} is a triple $G = (\Sigma,I,D)$, where $\Sigma$ is a finite alphabet, $I$ is a finite subset of $\Sigma \times \Typ(\Pri)$, and $D$ is a type.
A string $w = a_1 \dots a_n$ of length $n \geq 0$ is \emph{generated by} $G$ if there is a provable sequent $\Gamma \seqarrow D$ such that the yield of $\Gamma$ is $A_1 \dots A_n$ and for each $i=1,\dots,n$, $(a_i,A_i) \in I$.
We write $L(G)$ for the set $\{\, w \in \Sigma\kleene \mid \text{$G$ generates $w$} \,\}$.
A language generated by some $\mathbf{L}\Diamond$ grammar is said to be \emph{recognized by} $\mathbf{L}\Diamond$.%
\footnote{This is one of the two notions of recognition studied by \cite{Jaeger:2003}; he called this notion \emph{t-recognition}.}
Since the antecedent of a provable sequent is never empty and never contains $\langle \rangle$, $\mathbf{L}\Diamond$ only recognizes languages consisting of nonempty strings (\emph{$\epsilon$-free} languages).

\cite{Jaeger:2003} claimed that $\mathbf{L}\Diamond$ recognizes exactly the ($\epsilon$-free) context-free languages.
His proposed proof relied on the following translation from types of $\mathbf{L}\Diamond$ to types of the original Lambek calculus $\mathbf{L}$ due to \cite{Versmissen:1996}:
\begin{align*}
p^{\flat} & = p,\\
(A \bs B)^{\flat} & = A^{\flat} \bs B^{\flat},\\
(B / A)^{\flat} & = B^{\flat} / A^{\flat},\\
(A \bullet B)^{\flat} & = A^{\flat} \bullet B^{\flat},\\
(\Diamond A)^{\flat} & = \textsf{m} \bullet A^{\flat} \bullet \textsf{n},\\
(\varbox A)^{\flat} & = \textsf{m} \bs A^{\flat} / \textsf{n},
\end{align*}
where $\textsf{m}$ and $\textsf{n}$ are new primitive types.
As pointed out by \cite{Fadda-and-Morrill:2005}, however, Versmissen's translation is not a faithful embedding in the sense that there is a sequent $A_1 \dots A_n \seqarrow B$ which is not provable in $\mathbf{L}\Diamond$ but whose translation, $A_1^{\flat} \dots A_n^{\flat} \seqarrow B^{\flat}$, is provable in $\mathbf{L}$.%
\footnote{An example (adapted from \cite{Fadda-and-Morrill:2005}) is $\Diamond \varbox p \concat \Diamond \varbox q \seqarrow \Diamond \varbox (p \bullet q)$.}
Consequently, J\"ager's proof does not go through.
According to \cite{Kanovich-et-al:2017}, it has remained an open question whether $\mathbf{L}\Diamond$ recognizes exactly the ($\epsilon$-free) context-free languages.%
\footnote{To be precise, \cite{Kanovich-et-al:2017} were speaking of $\mathbf{L}\kleene\Diamond$ rather than $\mathbf{L}\Diamond$.
\cite{Jaeger:2003} was dealing with $\mathbf{L}\Diamond$ rather than $\mathbf{L}\kleene\Diamond$, although he did not make it entirely clear.}

Fortunately, it is not necessary to rely on the faithfulness of Versmissen's translation to prove J\"ager's claim.
As we see below, a straightforward adaptation of the method from \cite{Pentus:1993,Pentus:1997} can be used to establish J\"ager's claim.

There are three main ingredients to Pentus's \citeyearpar{Pentus:1993,Pentus:1997} proof:
\begin{itemize}
\item
interpolation theorem for $\mathbf{L}$ (originally proved by \cite{Roorda:1991} for $\mathbf{L}\kleene$, the Lambek calculus allowing empty antecedents)
\item
soundness of the free group interpretation
\item
little lemma about free groups
\end{itemize}

We need the extension of the first two ingredients to the case of $\mathbf{L}\Diamond$.
An interpolation theorem for $\mathbf{L}\Diamond$ was proved by \cite{Jaeger:2003}.
The required free group interpretation for $\mathbf{L}\Diamond$ can be obtained through Versmissen's \citeyearpar{Versmissen:1996} translation; the faithfulness of the translation is not necessary.

In order to make use of his lemma about free groups, \cite{Pentus:1993,Pentus:1997} relied on the notion of a \emph{thin sequent}.
This is not essential; if we use links connecting positive and negative occurrences of primitive types instead of the free group interpretation, we can avoid the notion of a thin sequent.%
\footnote{See \cite{Kanazawa:2006} for a statement of an interpolation theorem for the implicational fragment of intuitionistic logic in terms of these links.}
Similar links that also connect occurrence of brackets and modalities can be used to reason about $\mathbf{L}\Diamond$ as well.
Nevertheless, both because of its convenience and because it allows us to stay close to Pentus's \citeyearpar{Pentus:1993,Pentus:1997} proof, we introduce a notion of a thin sequent appropriate for $\mathbf{L}\Diamond$.
In order to do this, we have to extend the language and use brackets and modalities indexed by positive integers.

\section{The Multimodal Calculus $\mathbf{L}\Diamond_{\mathrm{m}}$}
\label{sec:multimodal}

We use brackets and modalities indexed by positive integers: $\langle_i, \rangle_i, \Diamond_i, \varbox_i$.
We write $\Typ_{\mathrm{m}}(\mathcal{B})$ for the smallest superset of $\mathcal{B}$ such that $A,B \in \Typ_{\mathrm{m}}(\mathcal{B})$ implies $A \bs B, B/A, A \bullet B, \Diamond_i A, \varbox_i A \in \Typ_{\mathrm{m}}(\mathcal{B})$ for each $i \geq 1$.
Elements of $\Typ_{\mathrm{m}}(\Pri)$ are called \emph{indexed types}.
The length $||A||$ of an indexed type $A$ is defined as before, where we add two for each occurrence of an indexed modality.

\emph{Indexed type trees} and \emph{indexed type hedges} are defined by induction as follows:
\begin{itemize}
\item
If $A$ is an indexed type, then $A$ is an indexed type tree.
\item
If $T_1,\dots,T_n$ ($n \geq 1$) are indexed type trees, then $T_1 \dots T_n$ is an indexed type hedge.
\item
If $\Gamma$ is an indexed type hedge, then $\langle_i \Gamma \rangle_i$ is an indexed type tree for any positive integer $i$.
\end{itemize}

The rules of the indexed variant $\mathbf{L}\Diamond_{\mathrm{m}}$ of $\mathbf{L}\Diamond$ are the same as those of $\mathbf{L}\Diamond$ except that the rules for the modalities are replaced by the following:
\begin{gather*}
\infer[(\Diamond_i{\seqarrow})]{\Gamma[\Diamond_i A] \seqarrow C}{
	\Gamma[\langle_i A \rangle_i] \seqarrow C
}
\qquad
\infer[({\seqarrow}\Diamond_i)]{\langle_i \Gamma \rangle_i \seqarrow \Diamond_i A}{
	\Gamma \seqarrow A
}
\\
\infer[(\varbox_i{\seqarrow})]{\Gamma[\langle_i \varbox_i A \rangle_i] \seqarrow C}{
	\Gamma[A] \seqarrow C
}
\qquad
\infer[({\seqarrow}\varbox_i)]{\Gamma \seqarrow \varbox_i A}{
	\langle_i \Gamma \rangle_i \seqarrow A
}
\end{gather*}

This calculus was presented briefly by \cite{Moortgat:1996} as a straightforward ``multimodal generalization'' of $\mathbf{L}\Diamond$.
Again, the cut rule is eliminable.

We interpret indexed types and type hedges as elements of the free group generated by $\Pri \cup \{\, {\langle}_i \mid i \geq 1 \,\} \cup \{\, {\rangle}_i \mid i \geq 1 \,\}$:
\begin{align*}
\den{p_i} & = p_i,\\
\den{A \bs B} & = \den{A}^{-1} \, \den{B},\\
\den{B / A} & = \den{B} \, \den{A}^{-1},\\
\den{A \bullet B} & = \den{A} \, \den{B},\\
\den{\Diamond_i A} & = {\langle}_i \, \den{A} \, {\rangle}_i,\\
\den{\varbox_i A} & = {\langle}_i^{-1} \, \den{A} \, {\rangle}_i^{-1},\\
\den{T_1 \dots T_n} & = \den{T_1} \dots \den{T_n},\\
\den{\langle_i \Gamma \rangle_i} & = {\langle}_i \, \den{\Gamma} \, {\rangle}_i.
\end{align*}

\begin{lemma}
\label{lem:free-group}
If\/ $\vdash_{\mathbf{L}\Diamond_{\mathrm{m}}} \Gamma \seqarrow C$, then\/ $\den{\Gamma} = \den{C}$.
\end{lemma}
\begin{proof}
Straightforward induction on the cut-free proof of $\Gamma \seqarrow C$.
\qed
\end{proof}

As in \cite{Pentus:1993,Pentus:1997}, we write $\sigma_i(A)$, $\sigma_i(\Gamma)$, $\sigma_i(\Gamma \seqarrow C)$, etc., for the number of occurrences of $p_i$ in $A$, $\Gamma$, $\Gamma \seqarrow C$, etc.
We let $\tau_i(A)$, $\tau_i(\Gamma)$, $\tau_i(\Gamma \seqarrow C)$, etc., denote the total number of occurrences of $\langle_i, \Diamond_i, \varbox_i$ in $A$, $\Gamma$, $\Gamma \seqarrow C$, etc.
(Note that since $\langle_i$ always occurs paired with $\rangle_i$, the number of occurrences of $\langle_i$ in the antecedent of a sequent is the same as the number of occurrences of $\rangle_i$ in it.)
Evidently, we always have
\[
||A|| = \sum_i \sigma_i(A) + 2 \sum_i \tau_i(A).
\] 
An indexed sequent $\Gamma \seqarrow C$ is \emph{thin} if for each $i$, $\sigma_i(\Gamma \seqarrow C) \leq 2$ and $\tau_i(\Gamma \seqarrow C) \leq 2$.

A \emph{primitive type substitution} is a function $\theta\colon \Pri \rightarrow \Pri$.
A (non-indexed) sequent $\Gamma \seqarrow C$ is a \emph{substitution instance} of an indexed sequent $\Gamma' \seqarrow C'$ if for some primitive type substitution $\theta$, the former is obtained from the latter by uniformly replacing each $p_i$ by $\theta(p_i)$ and replacing each indexed bracket and indexed modality by the corresponding non-indexed variant.
For example, if $p$ is a primitive type, $\langle \langle p \rangle \concat \Diamond p \bs p \rangle \seqarrow \varbox \Diamond \Diamond p$ is a substitution instance of $\langle_2 \langle_1 p_1 \rangle_1 \concat \Diamond_1 p_1 \bs p_2 \rangle_2 \seqarrow \varbox_3 \Diamond_3 \Diamond_2 p_2$.
This example illustrates the following lemma:
\begin{lemma}
If\/ $\vdash_{\mathbf{L}\Diamond} \Gamma \seqarrow C$, then\/ $\Gamma \seqarrow C$ is a substitution instance of some thin indexed sequent\/ $\Gamma' \seqarrow C'$ such that\/ $\vdash_{\mathbf{L}\Diamond_{\mathrm{m}}} \Gamma' \seqarrow C'$.
\end{lemma}
Such a thin indexed sequent is obtained from the proof of the original sequent using distinct primitive types for distinct instances of initial sequents and using distinct indices for distinct instances of $({\seqarrow}\Diamond)$ and of $(\varbox{\seqarrow})$.
For example, the $\mathbf{L}\Diamond$ proof
\[
\infer[({\seqarrow}\varbox)]{\langle \langle p \rangle \concat \Diamond p \bs p \rangle \seqarrow \varbox \Diamond \Diamond p}{
	\infer[({\seqarrow}\Diamond)]{\langle \langle \langle p \rangle \concat \Diamond p \bs p \rangle \rangle \seqarrow \Diamond \Diamond p}{
		\infer[({\seqarrow}\Diamond)]{\langle \langle p \rangle \concat \Diamond p \bs p \rangle \seqarrow \Diamond p}{
			\infer[(\bs{\seqarrow})]{\langle p \rangle \concat \Diamond p \bs p \seqarrow p}{
				\infer[({\seqarrow}\Diamond)]{\langle p \rangle \seqarrow \Diamond p}{
					p \seqarrow p
				}
				&
				p \seqarrow p
			}
		}
	}
}
\]
yields the $\mathbf{L}\Diamond_{\mathrm{m}}$ proof
\[
\infer[({\seqarrow}\varbox_3)]{\langle_2 \langle_1 p_1 \rangle_1 \concat \Diamond_1 p_1 \bs p_2 \rangle_2 \seqarrow \varbox_3 \Diamond_3 \Diamond_2 p_2}{
	\infer[({\seqarrow}\Diamond_3)]{\langle_3 \langle_2 \langle_1 p_1 \rangle_1 \concat \Diamond_1 p_1 \bs p_2 \rangle_2 \rangle_3 \seqarrow \Diamond_3 \Diamond_2 p_2}{
		\infer[({\seqarrow}\Diamond_2)]{\langle_2 \langle_1 p_1 \rangle_1 \concat \Diamond_1 p_1 \bs p_2 \rangle_2 \seqarrow \Diamond_2 p_2}{
			\infer[(\bs{\seqarrow})]{\langle_1 p_1 \rangle_1 \concat \Diamond_1 p_1 \bs p_2 \seqarrow p_2}{
				\infer[({\seqarrow}\Diamond_1)]{\langle_1 p_1 \rangle_1 \seqarrow \Diamond_1 p_1}{
					p_1 \seqarrow p_1
				}
				&
				p_2 \seqarrow p_2
			}
		}
	}
}
\]

J\"ager's \citeyearpar{Jaeger:2003} proof of his interpolation theorem for $\mathbf{L}\Diamond$ can be repeated for $\mathbf{L}\Diamond_{\mathrm{m}}$ to give the following statement:
\begin{theorem}
\label{thm:interpolation}
If\/ $\vdash_{\mathbf{L}\Diamond_{\mathrm{m}}} \Gamma[\Delta] \seqarrow C$, where\/ $\Delta$ is a nonempty type hedge, then there is a type\/ $E$ such that
\begin{enumerate}
\item[\upshape (i)]
$\vdash_{\mathbf{L}\Diamond_{\mathrm{m}}} \Delta \seqarrow E$,
\item[\upshape (ii)]
$\vdash_{\mathbf{L}\Diamond_{\mathrm{m}}} \Gamma[E] \seqarrow C$,
\item[\upshape (iii)]
$\sigma_i(E) \leq \min(\sigma_i(\Delta),\sigma_i(\Gamma[\contexthole] \seqarrow C))$ for each\/ $i$,
\item[\upshape (iv)]
$\tau_i(E) \leq \min(\tau_i(\Delta),\tau_i(\Gamma[\contexthole] \seqarrow C))$ for each\/ $i$.
\end{enumerate}
\end{theorem}
The type $E$ in the theorem is referred to as the \emph{interpolant} for $\Gamma[\Delta] \seqarrow C$ (relative to the ``partition'' $(\Delta;\Gamma[\contexthole])$ of $\Gamma[\Delta]$).
\begin{proof}
We repeat J\"ager's proof adapted to $L\Diamond_{\mathrm{m}}$ for the sake of convenience to the reader.
We write 
\[
\Gamma[\selected{\Delta}] \stackrel{E}{\seqarrow} C
\]
to mean that $E$ satisfies the conditions (i)--(iv) of the theorem for $\Gamma[\Delta] \seqarrow C$, relative to the partition $(\Delta;\Gamma[\contexthole])$ of $\Gamma[\Delta]$.
Such a type $E$ is found by induction on the cut-free proof $\mathcal{D}$ of $\Gamma[\Delta] \seqarrow C$, as follows.
It is a routine task to check that the conditions (i)--(iv) are satisfied.

\textit{Case 1.}
$\mathcal{D}$ is an initial sequent $p_i \seqarrow p_i$.
Then the only relevant partition of the antecedent is $(p_i;\contexthole)$.
\[
\selected{p_i} \stackrel{p_i}{\seqarrow} p_i
\]

\textit{Case 2.}
$\mathcal{D}$ ends in an application of $(\bs{\seqarrow})$.
There are six subcases to consider.
\begin{gather*}
\infer[(\bs{\seqarrow})]{\Delta'[\selected{\Delta''[\Gamma \concat A \bs B]}] \stackrel{E}{\seqarrow} C}{
	\Gamma \seqarrow A
	&
	\Delta'[\selected{\Delta''[B]}] \stackrel{E}{\seqarrow} C
}
\qquad
\infer[(\bs{\seqarrow})]{\Delta'[\Gamma' \concat \selected{\Gamma'' \concat A \bs B \concat \Pi}] \stackrel{E \bs F}{\seqarrow} C}{
	\selected{\Gamma'} \concat \Gamma'' \stackrel{E}{\seqarrow} A
	&
	\Delta'[\selected{B \concat \Pi}] \stackrel{F}{\seqarrow} C
}
\\
\infer[(\bs{\seqarrow})]{\Delta'[\selected{\Pi \concat \Gamma'} \concat \Gamma'' \concat A \bs B] \stackrel{E \bullet F}{\seqarrow} C}{
	\selected{\Gamma'} \concat \Gamma'' \stackrel{F}{\seqarrow} A
	&
	\Delta'[\selected{\Pi} \concat B] \stackrel{E}{\seqarrow} C
}
\qquad
\infer[(\bs{\seqarrow})]{\Delta[\Gamma' \concat \selected{\Gamma''} \concat \Gamma''' \concat A \bs B] \stackrel{E}{\seqarrow} C}{
	\Gamma' \concat \selected{\Gamma''} \concat \Gamma''' \stackrel{E}{\seqarrow} A
	&
	\Delta[B] \seqarrow C
}
\\
\infer[(\bs{\seqarrow})]{\Delta'[\Delta''[\selected{\Pi}] \concat \Delta'''[\Gamma \concat \concat A \bs B]] \stackrel{E}{\seqarrow} C}{
	\Gamma \seqarrow A
	&
	\Delta'[\Delta''[\selected{\Pi}] \concat \Delta'''[B]] \stackrel{E}{\seqarrow} C
}
\qquad
\infer[(\bs{\seqarrow})]{\Delta'[\Delta''[\Gamma \concat \concat A \bs B] \concat \Delta'''[\selected{\Pi}]] \stackrel{E}{\seqarrow} C}{
	\Gamma \seqarrow A
	&
	\Delta'[\Delta''[B] \concat \Delta'''[\selected{\Pi}]] \stackrel{E}{\seqarrow} C
}
\end{gather*}

\textit{Case 3.}
$\mathcal{D}$ ends in an application of $({\seqarrow}\bs)$.
\[
\infer[({\seqarrow}\bs)]{\Pi'[\selected{\Pi''}] \stackrel{E}{\seqarrow} A \bs B}{
A \concat \Pi'[\selected{\Pi''}] \stackrel{E}{\seqarrow} B
}
\]

\textit{Case 4.}
$\mathcal{D}$ ends in an application of $(/{\seqarrow})$.
This case is treated similarly to Case 2.

\textit{Case 5.}
$\mathcal{D}$ ends in an application of $({\seqarrow}/)$.
Similar to Case 3.

\textit{Case 6.}
$\mathcal{D}$ ends in an application of $(\bullet{\seqarrow})$.
There are three subcases to consider.
\begin{gather*}
\infer[(\bullet{\seqarrow})]{\Gamma'[\selected{\Gamma''[A \bullet B]}] \stackrel{E}{\seqarrow} C}{
	\Gamma'[\selected{\Gamma''[A \concat B]}] \stackrel{E}{\seqarrow} C
}
\qquad
\infer[(\bullet{\seqarrow})]{\Gamma'[\selected{\Gamma''} \concat \Gamma'''[A \bullet B]] \stackrel{E}{\seqarrow} C}{
	\Gamma'[\selected{\Gamma''} \concat \Gamma'''[A \concat B]] \stackrel{E}{\seqarrow} C
}
\\
\infer[(\bullet{\seqarrow})]{\Gamma'[\Gamma''[A \bullet B] \concat \selected{\Gamma'''}] \stackrel{E}{\seqarrow} C}{
	\Gamma'[\Gamma''[A \concat B] \concat \selected{\Gamma'''}] \stackrel{E}{\seqarrow} C
}
\end{gather*}

\textit{Case 7.}
$\mathcal{D}$ ends in an application of $({\seqarrow}\bullet)$.
There are three subcases to consider.
\begin{gather*}
\infer[({\seqarrow}\bullet)]{\Gamma'[\selected{\Pi}] \concat \Delta \stackrel{E}{\seqarrow} A \bullet B}{
	\Gamma'[\selected{\Pi}] \stackrel{E}{\seqarrow} A
	&
	\Delta \seqarrow B
}
\qquad
\infer[({\seqarrow}\bullet)]{\Gamma \concat \Delta'[\selected{\Pi}] \stackrel{E}{\seqarrow} A \bullet B}{
	\Gamma \seqarrow A
	&
	\Delta'[\selected{\Pi}] \stackrel{E}{\seqarrow} B
}
\\
\infer[({\seqarrow}\bullet)]{\Gamma' \concat \selected{\Gamma'' \concat \Delta'} \concat \Delta'' \stackrel{E \bullet F}{\seqarrow} A \bullet B}{
	\Gamma' \concat \selected{\Gamma''} \stackrel{E}{\seqarrow} A
	&
	\selected{\Delta'} \concat \Delta'' \stackrel{F}{\seqarrow} B
}
\end{gather*}

\textit{Case 8.}
$\mathcal{D}$ ends in an application of $(\Diamond_i{\seqarrow})$.
There are three subcases to consider.
\begin{gather*}
\infer[(\Diamond_i{\seqarrow})]{\Gamma'[\selected{\Gamma''[\Diamond_i A]}] \stackrel{E}{\seqarrow} B}{
	\Gamma'[\selected{\Gamma''[\langle_i A \rangle_i]}] \stackrel{E}{\seqarrow} B
}
\qquad
\infer[(\Diamond_i{\seqarrow})]{\Gamma'[\Gamma''[\selected{\Pi}] \concat \Gamma'''[\Diamond_i A]] \stackrel{E}{\seqarrow} B}{
	\Gamma'[\Gamma''[\selected{\Pi}] \concat \Gamma'''[\langle_i A \rangle_i]] \stackrel{E}{\seqarrow} B
}
\\
\infer[(\Diamond_i{\seqarrow})]{\Gamma'[\Gamma''[\Diamond_i A] \concat \Gamma'''[\selected{\Pi}]] \stackrel{E}{\seqarrow} B}{
	\Gamma'[\Gamma''[\langle_i A \rangle_i] \concat \Gamma'''[\selected{\Pi}]] \stackrel{E}{\seqarrow} B
}
\end{gather*}

\textit{Case 9.}
$\mathcal{D}$ ends in an application of $({\seqarrow}\Diamond_i)$.
There are two subcases to consider.
\begin{gather*}
\infer[({\seqarrow}\Diamond_i)]{\selected{\langle_i \Gamma \rangle_i} \stackrel{\Diamond_i E}{\seqarrow} \Diamond_i A}{
	\selected{\Gamma} \stackrel{E}{\seqarrow} A
}
\qquad
\infer[({\seqarrow}\Diamond_i)]{\langle_i \Gamma'[\selected{\Pi}] \rangle_i \stackrel{E}{\seqarrow} \Diamond_i A}{
	\Gamma'[\selected{\Pi}] \stackrel{E}{\seqarrow} A
}
\end{gather*}
Note that in the first subcase, $\Gamma$ cannot be empty, so the induction hypothesis applies.

\textit{Case 10.}
$\mathcal{D}$ ends in an application of $(\varbox_i{\seqarrow})$.
There are four subcases to consider.
(For the first subcase, note that $\vdash_{\mathbf{L}\Diamond_{\mathrm{m}}} A \seqarrow E$ implies $\vdash_{\mathbf{L}\Diamond_{\mathrm{m}}} \varbox_i A \seqarrow \varbox_i E$.)
\begin{gather*}
\infer[(\varbox_i{\seqarrow})]{\Gamma[\langle_i \selected{\varbox_i A} \rangle_i] \stackrel{\varbox_i E}{\seqarrow} B}{
	\Gamma[\selected{A}] \stackrel{E}{\seqarrow} B
}
\qquad
\infer[(\varbox_i{\seqarrow})]{\Gamma'[\selected{\Gamma''[\langle_i \varbox_i A \rangle_i]}] \stackrel{E}{\seqarrow} B}{
	\Gamma'[\selected{\Gamma''[A]}] \stackrel{E}{\seqarrow} B
}
\\
\infer[(\varbox_i{\seqarrow})]{\Gamma'[\selected{\Pi} \concat \Gamma''[\langle_i \varbox_i A \rangle_i]] \stackrel{E}{\seqarrow} B}{
	\Gamma'[\selected{\Pi} \concat \Gamma''[A]] \stackrel{E}{\seqarrow} B
}
\qquad
\infer[(\varbox_i{\seqarrow})]{\Gamma'[\Gamma''[\langle_i \varbox_i A \rangle_i] \concat \selected{\Pi}] \stackrel{E}{\seqarrow} B}{
	\Gamma'[\Gamma''[A] \concat \selected{\Pi}] \stackrel{E}{\seqarrow} B
}
\end{gather*}

\textit{Case 11.}
$\mathcal{D}$ ends in an application of $({\seqarrow}\varbox_i)$.
\[
\infer[({\seqarrow}\varbox_i)]{\Gamma'[\selected{\Pi}] \stackrel{E}{\seqarrow} \varbox_i A}{
	\langle_i \Gamma'[\selected{\Pi}] \rangle_i \stackrel{E}{\seqarrow} A
}
\qedhere
\]
\end{proof}
Note that just as in the case of the interpolation theorem for $\mathbf{L}$, the proof of Theorem~\ref{thm:interpolation} gives an algorithm for computing cut-free proofs of $\Delta \seqarrow E$ and of $\Gamma[E] \seqarrow C$ from the given cut-free proof of $\Gamma[\Delta] \seqarrow C$.

Each element $u$ of the free group generated by some set $S$ has a unique shortest representation as the product of some elements of $S \cup \{\, a^{-1} \mid a \in S \,\}$.
The length of this shortest representation is denoted by $|u|$.
It is easy to see that we always have $|\den{A}| \leq ||A||$.
Suppose that $\Gamma[\Delta] \seqarrow C$ in Theorem~\ref{thm:interpolation} is a thin indexed sequent.
Then since $\sigma_i(\Delta) + \sigma_i(\Gamma[\contexthole] \seqarrow C) = \sigma_i(\Gamma[\Delta] \seqarrow C) \leq 2$ and $\tau_i(\Delta) + \tau_i(\Gamma[\contexthole] \seqarrow C) = \tau_i(\Gamma[\Delta] \seqarrow C) \leq 2$, it follows that the interpolant $E$ satisfies $\sigma_i(E) \leq 1$ and $\tau_i(E) \leq 1$.
As \cite{Pentus:1993,Pentus:1997} observed for the case of $\mathbf{L}$, this implies $||E|| = |\den{E}|$ and together with Lemma~\ref{lem:free-group} gives:
\begin{equation}
\label{eqn:thin-free-group}
||E|| = |\den{\Delta}|.
\end{equation}

The following little lemma played a crucial role in Pentus's \citeyearpar{Pentus:1993,Pentus:1997} proof:
\begin{lemma}[Pentus]
\label{lem:Pentus}
If\/ $u_1,\dots,u_n$\/ $(n \geq 2)$ are elements of the free group generated by some set such that\/ $u_1 \dots u_n$ equals the identity, then there is a number\/ $k < n$ such that\/ $|u_k u_{k+1}| \leq \max(|u_k|,|u_{k+1}|)$.
\end{lemma}

\section{The Recognizing Power of $\mathbf{L}\Diamond$}
\label{sec:LDiamond}

Let $\mathcal{S}$ be some finite set of sequents.
We write $\mathcal{S} \vdash_{\mathrm{Cut}} \Gamma \seqarrow A$ to mean that the sequent $\Gamma \seqarrow A$ can be derived from $\mathcal{S}$ using Cut only.
Let $\mathcal{B}$ be a finite set of primitive types, and define
\[
\mathcal{S}_{\mathcal{B},m} = \{\, A_1 \dots A_n \seqarrow A_{n+1} \mid 
\begin{aligned}[t]
& n \leq 2,\\
& \text{$A_i \in \Typ(\mathcal{B})$ and $||A_i|| \leq m$ $(1 \leq i \leq n+1)$,}\\
& {\vdash_{\mathbf{L}\Diamond} A_1 \dots A_n \seqarrow A_{n+1}} \,\}.
\end{aligned}
\]
Clearly, $\mathcal{S}_{\mathcal{B},m}$ is finite.
Combining Lemma~\ref{lem:Pentus} with Theorem~\ref{thm:interpolation} and Lemma~\ref{lem:free-group} in the exact same way as \cite{Pentus:1993} did with the corresponding results about $\mathbf{L}$, we can prove the following:
\begin{lemma}
\label{lem:without-brackets}
Suppose\/ $A_i \in \Typ(\mathcal{B})$ and\/ $||A_i|| \leq m$ for\/ $i=1,\dots,n+1$.
Then\/ $\vdash_{\mathbf{L}\Diamond} A_1 \dots A_n \seqarrow A_{n+1}$ only if\/ $\mathcal{S}_{\mathcal{B},m} \vdash_{\mathrm{Cut}} A_1 \dots A_n \seqarrow A_{n+1}$.
\end{lemma}
\begin{proof}
Induction on $n$.
If $n \leq 2$, then $A_1 \dots A_n \seqarrow A_{n+1}$ is in $\mathcal{S}_{\mathcal{B},m}$, so $\mathcal{S}_{\mathcal{B},m} \vdash_{\mathrm{Cut}} A_1 \dots A_n \seqarrow A_{n+1}$.
If $n \geq 3$, let $A_1' \dots A_n' \seqarrow A'_{n+1}$ be a thin indexed sequent such that $\vdash_{\mathbf{L}\Diamond_{\mathrm{m}}} A_1' \dots A_n' \seqarrow A'_{n+1}$ and $A_1 \dots A_n \seqarrow A_{n+1}$ can be obtained by applying the substitution $\theta$ to the primitive types and removing all subscripts from the modalities in $A_1' \dots A_n' \seqarrow A_{n+1}'$.
Let $u_i = \den{A_i'}$ for $i=1,\dots,n+1$.
Since $u_1 \dots u_n = u_{n+1}$ by Lemma~\ref{lem:free-group}, $u_1 \dots u_n u_{n+1}^{-1}$ equals the identity.
Since $|\den{A_i'}| \leq ||A_i'|| \leq m$, we clearly have $|u_i| \leq m$ for $i=1,\dots,n$ and $|u_{n+1}^{-1}| = |u_{n+1}| \leq m$.
By Lemma~\ref{lem:Pentus}, either $|u_k u_{k+1}| \leq m$ for some $k \leq n-1$ or $|u_n u_{n+1}^{-1}| \leq m$.

\textit{Case 1.}
$|u_k u_{k+1}| \leq m$ for some $k \leq n-1$.
Let $E'$ be the interpolant for $A_1' \dots A_n' \seqarrow A_{n+1}'$ with respect to the partition $(A_k' A_{k+1}'; A_1' \dots A_{k-1}' \concat \contexthole \concat A_{k+2}' \dots A_n')$ of its antecedent.
By the remark following Theorem~\ref{thm:interpolation} (equation (\ref{eqn:thin-free-group})), $||E'|| = |\den{A_k' A_{k+1}'}| = |u_k u_{k+1}| \leq m$.
Let $E$ be the result of applying the substitution $\theta$ to the primitive types and removing subscripts from the modalities in $E'$.
Since $\vdash_{\mathbf{L}\Diamond_{\mathrm{m}}} A_k' A_{k+1}' \seqarrow E'$ and $\vdash_{\mathbf{L}\Diamond_{\mathrm{m}}} A_1' \dots A_{k-1}' E' A_{k+2}' \dots A_n' \seqarrow A_{n+1}'$, we must have $\vdash_{\mathbf{L}\Diamond} A_k A_{k+1} \seqarrow E$ and $\vdash_{\mathbf{L}\Diamond} A_1 \dots A_{k-2} E A_{k+2} \dots A_n \seqarrow A_{n+1}$.
Since $||E|| = ||E'|| \leq m$, $A_k A_{k+1} \seqarrow E$ is in $\mathcal{S}_{\mathcal{B},m}$.
By the induction hypothesis, $\mathcal{S}_{\mathcal{B},m} \vdash_{\mathrm{Cut}} A_1 \dots A_{k-2} E A_{k+2} \dots A_n \seqarrow A_{n+1}$.
It follows that $\mathcal{S}_{\mathcal{B},m} \vdash_{\mathrm{Cut}} A_1 \dots A_n \seqarrow A_{n+1}$.

\textit{Case 2.}
$|u_n u_{n+1}^{-1}| \leq m$.
Since $u_1 \dots u_{n-1} = (u_n u_{n+1}^{-1})^{-1}$, we have $|u_1 \dots u_{n-1}| \leq m$.
Let $E'$ be the interpolant for $A_1' \dots A_n' \seqarrow A_{n+1}'$ with respect to the partition $(A_1' \dots A_{n-1}'; \,\contexthole \concat A_n')$ of its antecedent.
As in Case 1, we have $||E'|| = |\den{A_1' \dots A_{n-1}'}| = |u_1 \dots u_{n-1}| \leq m$.
Let $E$ be the result of applying the substitution $\theta$ to the primitive types and removing subscripts from the modalities in $E'$.
Since $\vdash_{\mathbf{L}\Diamond_{\mathrm{m}}} A_1' \dots A_{n-1}' \seqarrow E'$ and $\vdash_{\mathbf{L}\Diamond_{\rm{m}}} E' A_n' \seqarrow A_{n+1}'$, we must have $\vdash_{\mathbf{L}\Diamond} A_1 \dots A_{n-1} \seqarrow E$ and $\vdash_{\mathbf{L}\Diamond} E A_n \seqarrow A_{n+1}$.
Since $||E|| = ||E'|| \leq m$, the sequent $E A_n \seqarrow A_{n+1}$ is in $\mathcal{S}_{\mathcal{B},m}$, and $\mathcal{S}_{\mathcal{B},m} \vdash_{\mathrm{Cut}} A_1 \dots A_{n-1} \seqarrow E$ by induction hypothesis.
It follows that $\mathcal{S}_{\mathcal{B},m} \vdash_{\mathrm{Cut}} A_1 \dots A_n \seqarrow A_{n+1}$.
\qed
\end{proof}

Lemma~\ref{lem:without-brackets} only takes care of $\mathbf{L}\Diamond$-provable sequents without brackets.
We need to find a finite set of sequents $\mathcal{T}_{\mathcal{B},m}$ such that if $\vdash_{\mathbf{L}\Diamond} \Gamma \seqarrow A_{n+1}$, the yield of $\Gamma$ is $A_1 \dots A_n$, and $||A_i|| \leq m$ for $i=1,\dots,n+1$, then $\mathcal{T}_{\mathcal{B},m} \vdash_{\mathrm{Cut}} \Gamma \seqarrow A_{n+1}$.
The following definition will do:
\[
\mathcal{T}_{\mathcal{B},m} = 
\begin{aligned}[t]
& \mathcal{S}_{\mathcal{B},m} \cup {}\\
& \{\, \langle A \rangle \seqarrow \Diamond A \mid A \in \Typ(\mathcal{B}), ||A|| \leq m-2 \,\} \cup {}\\
& \{\, \langle \varbox A \rangle \seqarrow A \mid A \in \Typ(\mathcal{B}), ||A|| \leq m-2 \,\}.
\end{aligned}
\]
\citet[Lemma~7.5]{Jaeger:2003} came very close to showing that $\mathcal{T}_{\mathcal{B},m}$ satisfies the required property, but incorrectly relied on the faithfulness of Versmissen's \citeyearpar{Versmissen:1996} translation.

\cite{Jaeger:2003} derived the following as a consequence of his interpolation theorem for $\mathbf{L}\Diamond$:
\begin{lemma}[J\"ager]
\label{lem:Jaeger}
Suppose\/ $\vdash_{\mathbf{L}\Diamond} \Gamma[\langle \Delta \rangle] \seqarrow A_{n+1}$, where the yield of\/ $\Gamma[\langle \Delta \rangle]$ is\/ $A_1 \dots A_n$ with\/ $A_i \in \Typ(\mathcal{B})$ and\/ $||A_i|| \leq m$ for\/ $i=1,\dots,n+1$.
Then there is a\/ $B \in \Typ(\mathcal{B})$ such that\/ $||B|| \leq m-2$ and one of the following holds:
\begin{enumerate}
\item[\upshape (i)]
$\vdash_{\mathbf{L}\Diamond} \Delta \seqarrow B$ and\/ $\vdash_{\mathbf{L}\Diamond} \Gamma[\Diamond B] \seqarrow A_{n+1}$.
\item[\upshape (ii)]
$\vdash_{\mathbf{L}\Diamond} \Delta \seqarrow \varbox B$ and\/ $\vdash_{\mathbf{L}\Diamond} \Gamma[B] \seqarrow A_{n+1}$.
\end{enumerate}
\end{lemma}

This together with Lemma~\ref{lem:without-brackets} is enough to establish the following:
\begin{lemma}
\label{lem:cut-completeness}
Let\/ $\Gamma \seqarrow A_{n+1}$ be an\/ $\mathbf{L}\Diamond$ sequent such that the yield of\/ $\Gamma$ is\/ $A_1 \dots A_n$ with\/ $A_i \in \Typ(\mathcal{B})$ and\/ $||A_i|| \leq m$ for\/ $i=1,\dots,n+1$.
Then\/ $\vdash_{\mathbf{L}\Diamond} \Gamma \seqarrow A_{n+1}$ if and only if\/ $\mathcal{T}_{\mathcal{B},m} \vdash_{\mathrm{Cut}} \Gamma \seqarrow A_{n+1}$.
\end{lemma}
\begin{proof}
Since $\vdash_{\mathbf{L}\Diamond} \langle A \rangle \seqarrow \Diamond A$ and $\vdash_{\mathbf{L}\Diamond} \langle \varbox A \rangle \seqarrow A$ for any $A$, the ``if'' direction is immediate.

For the ``only if'' direction, suppose $\vdash_{\mathbf{L}\Diamond} \Gamma \seqarrow A_{n+1}$.
We reason by induction on the number of occurrences of brackets in $\Gamma$.
If no bracket occurs in $\Gamma$, then $\Gamma = A_1 \dots A_n$ and it follows from Lemma~\ref{lem:without-brackets} that $\mathcal{T}_{\mathcal{B},m} \vdash_{\mathrm{Cut}} \Gamma \seqarrow A_{n+1}$.
If $\Gamma = \Gamma'[\langle \Delta \rangle]$, then we can apply Lemma~\ref{lem:Jaeger} and obtain a type $B \in \Typ(\mathcal{B})$ with $||B|| \leq m-2$ such that either (i) $\vdash_{\mathbf{L}\Diamond} \Delta \seqarrow B$ and $\vdash_{\mathbf{L}\Diamond} \Gamma'[\Diamond B] \seqarrow A_{n+1}$ or (ii) $\vdash_{\mathbf{L}\Diamond} \Delta \seqarrow \varbox B$ and $\vdash_{\mathbf{L}\Diamond} \Gamma'[B] \seqarrow A_{n+1}$.
Note that 
\[
\{ \Delta \seqarrow B,\, \Gamma'[\Diamond B] \seqarrow A_{n+1},\, \langle B \rangle \seqarrow \Diamond B \} \vdash_{\mathrm{Cut}} \Gamma'[\langle \Delta \rangle] \seqarrow A_{n+1}
\]
and
\[
\{ \Delta \seqarrow \varbox B,\, \Gamma'[B] \seqarrow A_{n+1},\, \langle \varbox B \rangle \seqarrow B \} \vdash_{\mathrm{Cut}} \Gamma'[\langle \Delta \rangle] \seqarrow A_{n+1}.
\]
In the case of (i), since both $\Delta \seqarrow B$ and $\Gamma'[\Diamond B] \seqarrow A_{n+1}$ contain fewer brackets than $\Gamma'[\langle \Delta \rangle] \seqarrow A_{n+1}$, the induction hypothesis implies that $\mathcal{T}_{\mathcal{B},m} \vdash_{\mathrm{Cut}} \Delta \seqarrow B$ and $\mathcal{T}_{\mathcal{B},m} \vdash_{\mathrm{Cut}} \Gamma'[\Diamond B] \seqarrow A_{n+1}$.
Since $\langle B \rangle \seqarrow \Diamond B$ is in $\mathcal{T}_{\mathcal{B},m}$, it follows that $\mathcal{T}_{\mathcal{B},m} \vdash_{\mathrm{Cut}} \Gamma'[\langle \Delta \rangle] \seqarrow A_{n+1}$.
Similarly, in the case of (ii), since both $\Delta \seqarrow \varbox B$ and $\Gamma'[B] \seqarrow A_{n+1}$ contain fewer brackets than $\Gamma'[\langle \Delta \rangle] \seqarrow A_{n+1}$, the induction hypothesis gives $\mathcal{T}_{\mathcal{B},m} \vdash_{\mathrm{Cut}} \Delta \seqarrow \varbox B$ and $\mathcal{T}_{\mathcal{B},m} \vdash_{\mathrm{Cut}} \Gamma'[B] \seqarrow A_{n+1}$.
Since $\langle \varbox B \rangle \seqarrow B$ is in $\mathcal{T}_{\mathcal{B},m}$, it follows that $\mathcal{T}_{\mathcal{B},m} \vdash_{\mathrm{Cut}} \Gamma'[\langle \Delta \rangle] \seqarrow A_{n+1}$.
\qed
\end{proof}

\begin{theorem}
\label{thm:LDiamond-context-free}
Every language recognized by\/ $\mathbf{L}\Diamond$ is context-free.
\end{theorem}
\begin{proof}
Let $G = (\Sigma,I,D)$ be an $\mathbf{L}\Diamond$ grammar.
Let $\mathcal{B}$ be the set of primitive types used in $G$, and let
\[
m = \max(\{\, ||A|| \mid \text{$(a,A) \in I$ for some $a \in \Sigma$} \,\} \cup \{ ||D|| \}).
\]
Define a context-free grammar $G' = (N,\Sigma,P,D)$ by
\begin{align*}
N & = \{\, A \in \Typ(\mathcal{B}) \mid ||A|| \leq m \,\},\\
P & = 
\begin{aligned}[t]
& \{\, A_{n+1} \rightarrow A_1 \dots A_n \mid \text{$A_1 \dots A_n \seqarrow A_{n+1}$ is in $\mathcal{S}_{\mathcal{B},m}$} \,\} \cup {}\\
& \{\, \Diamond A \rightarrow A \mid \text{$A \in \Typ(\mathcal{B})$ and $||A|| \leq m-2$} \,\} \cup {}\\
& \{\, A \rightarrow \varbox A \mid \text{$A \in \Typ(\mathcal{B})$ and $||A|| \leq m-2$} \,\} \cup {}\\
& \{\, A \rightarrow a \mid (a,A) \in I \,\}.
\end{aligned}
\end{align*}
We prove that $G$ and $G'$ generate the same language.
It is clearly enough to prove that the following are equivalent whenever $A_i \in \Typ(\mathcal{B})$ and $||A_i|| \leq m$ for $i=1,\dots,n+1$:
\begin{enumerate}
\item[(i)]
$\vdash_{\mathbf{L}\Diamond} \Gamma \seqarrow A_{n+1}$ for some $\Gamma$ whose yield is $A_1 \dots A_n$.
\item[(ii)]
$A_{n+1} \Rightarrow_{G'}\kleene A_1 \dots A_n$.
\end{enumerate}
By Lemma~\ref{lem:cut-completeness}, (i) is equivalent to
\begin{enumerate}
\item[(i$'$)]
$\mathcal{T}_{\mathcal{B},m} \vdash_{\mathrm{Cut}} \Gamma \seqarrow A_{n+1}$ for some $\Gamma$ whose yield is $A_1 \dots A_n$.
\end{enumerate}
That (i$'$) implies (ii) is proved by straightforward induction on the number of applications of Cut to derive $\Gamma \seqarrow A_{n+1}$ from $\mathcal{T}_{\mathcal{B},m}$.
The converse implication is proved by equally straightforward induction on the length of the derivation of $A_{n+1} \Rightarrow_{G'}\kleene A_1 \dots A_n$.
\qed
\end{proof}

\section{The Calculus $\mathbf{L}\kleene\Diamond$}
\label{sec:L*b}

The calculus $\mathbf{L}\kleene\Diamond$ consists of the rules of $\mathbf{L}\Diamond$ without the restriction on $({\seqarrow}\bs)$ and $({\seqarrow}/)$.
The multimodal variant is $\mathbf{L}\kleene\Diamond_{\mathrm{m}}$.
The method of Sections~\ref{sec:multimodal} and~\ref{sec:LDiamond} is not directly applicable to $\mathbf{L}\kleene\Diamond$.
This is because the interpolation theorem (Theorem~\ref{thm:interpolation}) does not hold of $\mathbf{L}\kleene\Diamond_{\mathrm{m}}$ (or of $\mathbf{L}\kleene\Diamond$, for that matter).
For example, we have
\begin{equation}
\label{eqn:no-interpolant}
\vdash_{\mathbf{L}\kleene\Diamond_{\mathrm{m}}} p_3/\Diamond_1(p_1 \bullet \Diamond_2(p_2/p_2)) \concat \langle_1 p_1 \concat \langle_2 \rangle_2 \rangle_1 \seqarrow p_3,
\end{equation}
but there is no type $E$ such that
\begin{gather*}
\vdash_{\mathbf{L}\kleene\Diamond_{\mathrm{m}}} \langle_1 p_1 \concat \langle_2 \rangle_2 \rangle_1 \seqarrow E,\\
\vdash_{\mathbf{L}\kleene\Diamond_{\mathrm{m}}} p_3/\Diamond_1(p_1 \bullet \Diamond_2(p_2/p_2)) \concat E \seqarrow p_3,\\
\text{$\sigma_1(E) \leq 1$ and $\sigma_i(E) = 0$ for $i \geq 2$,}\\
\text{$\tau_1(E) \leq 1$, $\tau_2(E) \leq 1$, and $\tau_i(E) = 0$ for $i \geq 3$.}
\end{gather*}
To see this, note that Lemma~\ref{lem:free-group} holds of $\mathbf{L}\kleene\Diamond_{\mathrm{m}}$ as well and implies $\den{E} = \langle_1 p_1 \langle_2 \rangle_2 \rangle_1$, but $E$ can contain no more than one occurrence of an atomic type.
This is clearly impossible.

We can restore interpolation by adding the type constant $\bm{1}$ (the \emph{unit}) to the $\mathbf{L}\Diamond$ and $\mathbf{L}\Diamond_{\mathrm{m}}$ types, governed by the rules
\[
\infer[(\bm{1}{\seqarrow})]{\Gamma[\bm{1}] \seqarrow A}{
	\Gamma[] \seqarrow A
}
\qquad
\infer[({\seqarrow}\bm{1})]{\seqarrow \bm{1}}{}
\]
In $(\bm{1}{\seqarrow})$, $\Gamma[]$ is the result of replacing $\contexthole$ in $\Gamma[\contexthole]$ by the empty type hedge.
The resulting calculi are referred to as $\mathbf{L}\kleene_{\bm{1}}\Diamond$ and $\mathbf{L}\kleene_{\bm{1}}\Diamond_{\mathrm{m}}$.
(\cite{Pentus:1999} referred to the calculus $\mathbf{L}\kleene$ enriched with the unit as $\mathbf{L}\kleene_{\bm{1}}$.)
The types used in these calculi are the elements of $\Typ(\Pri \cup \{ \bm{1} \})$ and of $\Typ_{\mathrm{m}}(\Pri \cup \{ \bm{1} \})$, respectively.
Cut elimination holds of these calculi.%
\footnote{To extend Moortgat's \citeyearpar{Moortgat:1996} proof in the presence of $\bm{1}$, one only need to add the reduction step
\[
\infer[\text{Cut}]{\Gamma[] \seqarrow A}{
	\infer[({\seqarrow}\bm{1})]{\seqarrow \bm{1}}{}
	&
	\infer[(\bm{1}{\seqarrow})]{\Gamma[\bm{1}] \seqarrow A}{
		\infer*{\Gamma[] \seqarrow A}{}
	}
}
\quad
\rightsquigarrow
\quad
\infer*{\Gamma[] \seqarrow A}{}
\]
}

\begin{theorem}
\label{thm:interpolation-unit}
If\/ $\vdash_{\mathbf{L}\kleene_{\bm{1}}\Diamond_{\mathrm{m}}} \Gamma[\Delta] \seqarrow C$, then there is a type\/ $E$ such that
\begin{enumerate}
\item[\upshape (i)]
$\vdash_{\mathbf{L}\kleene_{\bm{1}}\Diamond_{\mathrm{m}}} \Delta \seqarrow E$,
\item[\upshape (ii)]
$\vdash_{\mathbf{L}\kleene_{\bm{1}}\Diamond_{\mathrm{m}}} \Gamma[E] \seqarrow C$,
\item[\upshape (iii)]
$\sigma_i(E) \leq \min(\sigma_i(\Delta),\sigma_i(\Gamma[\contexthole] \seqarrow C))$ for each\/ $i$,
\item[\upshape (iv)]
$\tau_i(E) \leq \min(\tau_i(\Delta),\tau_i(\Gamma[\contexthole] \seqarrow C))$ for each\/ $i$.
\end{enumerate}
\end{theorem}
\begin{proof}
Two new cases are handled as follows.
When $\Delta$ is the empty hedge, then we let $E = \bm{1}$.
When $\Delta = \bm{1}$ is introduced by $(\bm{1}{\seqarrow})$ at the last step, then we again let $E = \bm{1}$.
\qed
\end{proof}
For example, we can take $E = \Diamond_1(p_1 \bullet \Diamond_2 \bm{1})$ as the interpolant for the above example (\ref{eqn:no-interpolant}):%
\newlength{\temp}\settowidth{\temp}{$\langle$}
\[
\infer[(/{\seqarrow})]{p_3/\Diamond_1(p_1 \bullet \Diamond_2(p_2/p_2)) \concat \selected{\langle_1 p_1 \concat \langle_2 \rangle_2 \rangle_1} \stackrel{\Diamond_1(p_1 \bullet \Diamond_2 \bm{1})}{\seqarrow} p_3}{
	\infer[({\seqarrow}\Diamond_1)]{\selected{\langle_1 p_1 \concat \langle_2 \rangle_2 \rangle_1} \stackrel{\Diamond_1(p_1 \bullet \Diamond_2 \bm{1})}{\seqarrow} \Diamond_1(p_1 \bullet \Diamond_2(p_2/p_2))}{
		\infer[({\seqarrow}\bullet)]{\selected{p_1 \concat \langle_2 \rangle_2} \stackrel{p_1 \bullet \Diamond_2 \bm{1}}{\seqarrow} p_1 \bullet \Diamond_2(p_2/p_2)}{
			\selected{\phantom{\langle\hspace{-\temp}}p_1} \stackrel{p_1}{\seqarrow} p_1
			&
			\infer[({\seqarrow}\Diamond_2)]{\selected{\langle_2 \rangle_2} \stackrel{\Diamond_2 \bm{1}}{\seqarrow} \Diamond_2(p_2/p_2)}{
				\infer[({\seqarrow}/)]{\selected{\phantom{\langle\hspace{-\temp}}} \stackrel{\bm{1}}{\seqarrow} p_2/p_2}{
					\selected{\phantom{\langle\hspace{-\temp}}} \concat p_2 \stackrel{\bm{1}}{\seqarrow} p_2
				}
			}
		}
	}
	&
	p_3 \seqarrow p_3
}
\]

Naturally, we take $\den{\bm{1}}$ to be the identity element of the free group generated by $\Pri \cup \{\, {\langle}_i \mid i \geq 1 \,\} \cup \{\, {\rangle}_i \mid i \geq 1 \,\}$ so that Lemma~\ref{lem:free-group} continues to hold for $\mathbf{L}\kleene_{\bm{1}}\Diamond_{\mathrm{m}}$.
If we let $||\bm{1}|| = 0$ in the definition of $||A||$ for $\mathbf{L}\kleene_{\bm{1}}\Diamond_{\mathrm{m}}$ types, then whenever $\Gamma[\Delta] \seqarrow C$ is a thin sequent we again have equation (\ref{eqn:thin-free-group}) for the interpolant $E$ for this sequent.
Lemmas~\ref{lem:without-brackets}, \ref{lem:Jaeger}, and~\ref{lem:cut-completeness} continue to hold \emph{mutatis mutandis} for $\mathbf{L}\kleene_{\bm{1}}\Diamond$.
This does \emph{not}, however, imply that $\mathbf{L}\kleene_{\bm{1}}\Diamond$ (or $\mathbf{L}\kleene\Diamond$) only recognizes context-free languages.
The pitfall is that the sets $\mathcal{S}_{\mathcal{B},m}$ and $\mathcal{T}_{\mathcal{B},m}$ with $\Typ(\mathcal{B})$ replaced by $\Typ(\mathcal{B} \cup \{ \bm{1} \})$ are both infinite, since the conditions $||A_i|| \leq m$ and $||A|| \leq m-2$ in the definition of these sets place no bound on the number of occurrences of $\bm{1}$.

For instance, define types $A_i$ ($i=0,1,2,\dots$) by
\begin{align*}
A_0 & = q,\\
A_{i+1} & = (\bm{1}/A_i) \bs \bm{1},
\end{align*}
where $q$ is a primitive type.
It is easy to show by induction on $j$ that $\not\vdash_{\mathbf{L}\kleene_{\bm{1}}} A_i \seqarrow A_j$ whenever $i > j$.
So these are pairwise inequivalent types, but $||A_i|| = 1$ for all $i$.

We can see that the types $A_i$ even arise as interpolants for sequents consisting only of very short types.
Consider the cut-free proof:
\[
\infer[(\bs{\seqarrow})]{(\bm{1}/\bm{1})^{i-1} \concat \bm{1}/q \concat q \concat (\bm{1} \bs \bm{1})^i \seqarrow \bm{1}}{
	\infer[(/{\seqarrow})]{(\bm{1}/\bm{1})^{i-1} \concat \bm{1}/q \concat q \concat (\bm{1} \bs \bm{1})^{i-1} \seqarrow \bm{1}}{
		\infer*{(\bm{1}/\bm{1})^{i-2} \concat \bm{1}/q \concat q \concat (\bm{1} \bs \bm{1})^{i-1} \seqarrow \bm{1}}{
			\infer[(\bs{\seqarrow})]{\bm{1}/q \concat q \concat \bm{1} \bs \bm{1} \seqarrow \bm{1}}{
				\infer[(/{\seqarrow})]{\bm{1}/q \concat q \seqarrow \bm{1}}{
					q \seqarrow q
					&
					\infer[(\bm{1}{\seqarrow})]{\bm{1} \seqarrow \bm{1}}{
						\infer[({\seqarrow}\bm{1})]{\seqarrow \bm{1}}{
						}
					}
				}
				&
				\infer[(\bm{1}{\seqarrow})]{\bm{1} \seqarrow \bm{1}}{
					\infer[({\seqarrow}\bm{1})]{\seqarrow \bm{1}}{
					}
				}
			}
		}
		&
		\infer[(\bm{1}{\seqarrow})]{\bm{1} \seqarrow \bm{1}}{
			\infer[({\seqarrow}\bm{1})]{\seqarrow \bm{1}}{
			}
		}
	}
	&
	\infer[(\bm{1}{\seqarrow})]{\bm{1} \seqarrow \bm{1}}{
		\infer[({\seqarrow}\bm{1})]{\seqarrow \bm{1}}{
		}
	}
}
\]
The interpolant for $(\bm{1}/\bm{1})^{i-1} \concat \bm{1}/q \concat q \concat (\bm{1} \bs \bm{1})^i \seqarrow \bm{1}$ with respect to the partition $(q \concat (\bm{1} \bs \bm{1})^i; (\bm{1}/\bm{1})^{i-1} \concat \bm{1}/q \concat \contexthole)$ is computed from this proof by the method of Theorem~\ref{thm:interpolation-unit} as follows:
\[
\infer[(\bs{\seqarrow})]{(\bm{1}/\bm{1})^{i-1} \concat \bm{1}/q \concat \selected{q \concat (\bm{1} \bs \bm{1})^i} \stackrel{(\bm{1}/A_{i-1}) \bs \bm{1}}{\seqarrow} \bm{1}}{
	\infer[(/{\seqarrow})]{\selected{(\bm{1}/\bm{1})^{i-1} \concat \bm{1}/q} \concat q \concat (\bm{1} \bs \bm{1})^{i-1} \stackrel{\bm{1}/A_{i-1}}{\seqarrow} \bm{1}}{
		\infer*{(\bm{1}/\bm{1})^{i-2} \concat \bm{1}/q \concat \selected{q \concat (\bm{1} \bs \bm{1})^{i-1}} \stackrel{A_{i-1}}{\seqarrow} \bm{1}}{
			\infer[(\bs{\seqarrow})]{\bm{1}/q \concat \selected{q \concat \bm{1} \bs \bm{1}} \stackrel{(\bm{1}/q) \bs \bm{1}}{\seqarrow} \bm{1}}{
				\infer[(/{\seqarrow})]{\selected{\bm{1}/q} \concat q \stackrel{\bm{1}/q}{\seqarrow} \bm{1}}{
					\selected{q} \stackrel{q}{\seqarrow} q
					&
					\infer[(\bm{1}{\seqarrow})]{\selected{\bm{1}} \stackrel{\bm{1}}{\seqarrow} \bm{1}}{
						\infer[({\seqarrow}\bm{1})]{\seqarrow \bm{1}}{
						}
					}
				}
				&
				\infer[(\bm{1}{\seqarrow})]{\selected{\bm{1}} \stackrel{\bm{1}}{\seqarrow} \bm{1}}{
					\infer[({\seqarrow}\bm{1})]{\seqarrow \bm{1}}{
					}
				}
			}
		}
		&
		\infer[(\bm{1}{\seqarrow})]{\selected{\bm{1}} \stackrel{\bm{1}}{\seqarrow} \bm{1}}{
			\infer[({\seqarrow}\bm{1})]{\seqarrow \bm{1}}{
			}
		}
	}
	&
	\infer[(\bm{1}{\seqarrow})]{\selected{\bm{1}} \stackrel{\bm{1}}{\seqarrow} \bm{1}}{
		\infer[({\seqarrow}\bm{1})]{\seqarrow \bm{1}}{
		}
	}
}
\]

In the above computation, the type $A_i$ is obtained as the interpolant for a sequent with $2i+1$ types in the antecedent with respect to a partition that splits the antecedent into strings of types of roughly equal length.
Alternatively, $A_1,\dots,A_i$ may be obtained from the same sequent by iterating the computation of interpolants, as follows:
\begin{gather*}
(\bm{1}/\bm{1})^{i-1} \concat \bm{1}/q \concat \selected{q \concat \bm{1} \bs \bm{1}} \concat (\bm{1} \bs \bm{1})^{i-1} \stackrel{A_1}{\seqarrow} \bm{1}\\
(\bm{1}/\bm{1})^{i-1} \concat \bm{1}/q \concat \selected{A_1 \concat \bm{1} \bs \bm{1}} \concat (\bm{1} \bs \bm{1})^{i-2} \stackrel{A_2}{\seqarrow} \bm{1}\\
\vdots\\
(\bm{1}/\bm{1})^{i-1} \concat \bm{1}/q \concat \selected{A_{i-1} \concat \bm{1} \bs \bm{1}} \stackrel{A_i}{\seqarrow} \bm{1}
\end{gather*}
In this list of sequents, the ``boxed'' part always consists of two types.
A cut-free proof of each sequent in the list (except the first) is obtained through the computation of the interpolant for the preceding sequent in the list and looks as follows:
\[
\infer*{(\bm{1}/\bm{1})^{i-1} \concat \bm{1}/q \concat \selected{A_j \concat \bm{1} \bs \bm{1}} \concat (\bm{1} \bs \bm{1})^{i-j-1} \stackrel{(\bm{1}/A_j) \bs \bm{1}}{\seqarrow} \bm{1}}{
	\infer[(\bs{\seqarrow})]{(\bm{1}/\bm{1})^{j+1} \concat \bm{1}/q \concat \selected{A_j \concat \bm{1} \bs \bm{1}} \concat \bm{1} \bs \bm{1} \stackrel{(\bm{1}/A_j) \bs \bm{1}}{\seqarrow} \bm{1}}{
		\infer[(/{\seqarrow})]{(\bm{1}/\bm{1})^{j+1} \concat \bm{1}/q \concat \selected{A_j \concat \bm{1} \bs \bm{1}} \stackrel{(\bm{1}/A_j) \bs \bm{1}}{\seqarrow} \bm{1}}{
			\infer[(\bs{\seqarrow})]{(\bm{1}/\bm{1})^j \concat \bm{1}/q \concat \selected{A_j \concat \bm{1} \bs \bm{1}} \stackrel{(\bm{1}/A_j) \bs \bm{1}}{\seqarrow} \bm{1}}{
				\infer[(/{\seqarrow})]{\selected{(\bm{1}/\bm{1})^j \concat \bm{1}/q} \concat A_j \stackrel{\bm{1}/A_j}{\seqarrow} \bm{1}}{
					\infer*{(\bm{1}/\bm{1})^{j-1} \concat \bm{1}/q \concat \selected{A_j} \stackrel{A_j}{\seqarrow} \bm{1}}{}
					&
					\infer[(\bm{1}{\seqarrow})]{\selected{\bm{1}} \stackrel{\bm{1}}{\seqarrow} \bm{1}}{
						\infer[({\seqarrow}\bm{1})]{\seqarrow \bm{1}}{}
					}
				}
				&
				\infer[(\bm{1}{\seqarrow})]{\selected{\bm{1}} \stackrel{\bm{1}}{\seqarrow} \bm{1}}{
					\infer[({\seqarrow}\bm{1})]{\seqarrow \bm{1}}{}
				}
			}
			&
			\infer[(\bm{1}{\seqarrow})]{\bm{1} \seqarrow \bm{1}}{
				\infer[({\seqarrow}\bm{1})]{\seqarrow \bm{1}}{}
			}
		}
		&
		\infer[(\bm{1}{\seqarrow})]{\bm{1} \seqarrow \bm{1}}{
			\infer[({\seqarrow}\bm{1})]{\seqarrow \bm{1}}{}
		}
	}
}
\]

The above consideration shows that even the proof of context-freeness of $\mathbf{L}\kleene_{\bm{1}}$ requires further arguments than \cite{Pentus:1999} indicated; his brief remark \citep[Remark 5.13]{Pentus:1999} that the arguments used for the Lambek calculus $\mathbf{L}$ ``hold also for the Lambek calculus with the unit and the calculus $\mathbf{L}\kleene$'' and consequently ``the class of languages generated by categorial grammars based on any of these calculi coincides with the class of all context-free languages'' is not justified.%
\footnote{Pentus's \citeyearpar{Pentus:1999} claim of context-freeness of $\mathbf{L}\kleene$, as opposed to $\mathbf{L}\kleene_{\bm{1}}$, is immune to this criticism since an interpolation theorem similar to Theorem~\ref{thm:interpolation} does hold for $\mathbf{L}\kleene$ and there's no need to use $\bm{1}$ in converting an $\mathbf{L}\kleene$ grammar to a context-free grammar.
The same criticism does apply to his claim about grammars based on multiplicative cyclic linear logic (CLL).}
For this reason, \cite{Kuznetsov:2012} relied on a translation from $\mathbf{L}\kleene_{\bm{1}}$ sequents to $\mathbf{L}\kleene$ sequents to show that $\mathbf{L}\kleene_{\bm{1}}$ only recognizes context-free languages.

Let us return to our original concern.
We have seen that interpolation for $\mathbf{L}\kleene\Diamond_{\mathrm{m}}$ sequents generally requires the use of $\bm{1}$, but Pentus's method does not directly apply to the calculus containing $\bm{1}$, at least not without significant modifications.
Fortunately, however, we do not need the full power of $\mathbf{L}\kleene_{\bm{1}}\Diamond_{\mathrm{m}}$ for the purpose of proving the context-freeness of $\mathbf{L}\kleene\Diamond$.
The unit $\bm{1}$ \emph{is} needed, but its use can be limited to occurrences as the immediate subtype of a type of the form $\Diamond_i \bm{1}$.
We call elements of $\Typ_{\mathrm{m}}(\Pri \cup \{\, \Diamond_i \bm{1} \mid i \geq 1 \,\})$ or of $\Typ(\Pri \cup \{ \Diamond \bm{1} \})$ \emph{guarded} types.
We can prove the following:

\begin{theorem}
\label{thm:interpolation-unit-guarded}
Let\/ $\Gamma[\Delta] \seqarrow C$ be an\/ $\mathbf{L}\kleene_{\bm{1}}\Diamond_{\mathrm{m}}$ sequent such that the types occurring in it are all guarded and\/ $\Delta$ is a nonempty hedge.
If\/ $\vdash_{\mathbf{L}\kleene_{\bm{1}}\Diamond_{\mathrm{m}}} \Gamma[\Delta] \seqarrow C$, then there is a guarded type\/ $E$ such that
\begin{enumerate}
\item[\upshape (i)]
$\vdash_{\mathbf{L}\kleene_{\bm{1}}\Diamond_{\mathrm{m}}} \Delta \seqarrow E$,
\item[\upshape (ii)]
$\vdash_{\mathbf{L}\kleene_{\bm{1}}\Diamond_{\mathrm{m}}} \Gamma[E] \seqarrow C$,
\item[\upshape (iii)]
$\sigma_i(E) \leq \min(\sigma_i(\Delta),\sigma_i(\Gamma[\contexthole] \seqarrow C))$ for each\/ $i$,
\item[\upshape (iv)]
$\tau_i(E) \leq \min(\tau_i(\Delta),\tau_i(\Gamma[\contexthole] \seqarrow C))$ for each\/ $i$.
\end{enumerate}
\end{theorem}
\begin{proof}
When $\Delta = \langle_i \rangle_i$ or $\Delta = \Diamond_i \bm{1}$, we let $E = \Diamond_i \bm{1}$.
The rest of the proof proceeds as before.
\qed
\end{proof}

If $A$ is a guarded type with $||A|| \leq m$, then there cannot be more than $\lfloor m/2 \rfloor$ occurrences of $\bm{1}$ in it.
It follows that for any finite set $\mathcal{B}$ of primitive types, the set of types $A$ in $\Typ(\mathcal{B} \cup \{ \Diamond \bm{1} \})$ such that $||A|| \leq m$ is finite.
This means that we can modify the Pentus construction by using guarded types only.

Define
\begin{align*}
\mathcal{S}'_{\mathcal{B},m} & = \{\, A_1 \dots A_n \seqarrow A_{n+1} \mid 
\begin{aligned}[t]
& n \leq 2,\\
& \text{$A_i \in \Typ(\mathcal{B} \cup \{ \Diamond \bm{1} \})$ and $||A_i|| \leq m$ $(1 \leq i \leq n+1)$,}\\
& {\vdash_{\mathbf{L}\kleene_{\bm{1}}\Diamond} A_1 \dots A_n \seqarrow A_{n+1}} \,\},
\end{aligned}
\\
\mathcal{T}'_{\mathcal{B},m} & = 
\begin{aligned}[t]
& \mathcal{S}'_{\mathcal{B},m} \cup {}\\
& \{ \langle \rangle \seqarrow \Diamond \bm{1} \} \cup {}\\
& \{\, \langle A \rangle \seqarrow \Diamond A \mid A \in \Typ(\mathcal{B} \cup \{ \Diamond \bm{1} \}), ||A|| \leq m-2 \,\} \cup {}\\
& \{\, \langle \varbox A \rangle \seqarrow A \mid A \in \Typ(\mathcal{B} \cup \{ \Diamond \bm{1} \}), ||A|| \leq m-2 \,\}.
\end{aligned}
\end{align*}
These sets are finite.

\begin{lemma}
\label{lem:without-brackets-guarded}
Suppose\/ $A_i \in \Typ(\mathcal{B} \cup \{ \Diamond\bm{1} \})$ and\/ $||A_i|| \leq m$ for\/ $i=1,\dots,n+1$.
Then\/ $\vdash_{\mathbf{L}\kleene_{\bm{1}}\Diamond} A_1 \dots A_n \seqarrow A_{n+1}$ only if\/ $\mathcal{S}'_{\mathcal{B},m} \vdash_{\mathrm{Cut}} A_1 \dots A_n \seqarrow A_{n+1}$.
\end{lemma}

\begin{lemma}
\label{lem:Jaeger-guarded}
Suppose\/ $\vdash_{\mathbf{L}\kleene_{\bm{1}}\Diamond} \Gamma[\langle \Delta \rangle] \seqarrow A_{n+1}$, where\/ $\Delta$ is not the empty hedge and the yield of\/ $\Gamma[\langle \Delta \rangle]$ is\/ $A_1 \dots A_n$ with\/ $A_i \in \Typ(\mathcal{B} \cup \{ \Diamond \bm{1} \})$ and\/ $||A_i|| \leq m$ for\/ $i=1,\dots,n+1$.
Then there is a\/ $B \in \Typ(\mathcal{B} \cup \{ \Diamond \bm{1} \})$ such that\/ $||B|| \leq m-2$ and one of the following holds:
\begin{enumerate}
\item[\upshape (i)]
$\vdash_{\mathbf{L}\kleene_{\bm{1}}\Diamond} \Delta \seqarrow B$ and\/ $\vdash_{\mathbf{L}\kleene_{\bm{1}}\Diamond} \Gamma[\Diamond B] \seqarrow A_{n+1}$.
\item[\upshape (ii)]
$\vdash_{\mathbf{L}\kleene_{\bm{1}}\Diamond} \Delta \seqarrow \varbox B$ and\/ $\vdash_{\mathbf{L}\kleene_{\bm{1}}\Diamond} \Gamma[B] \seqarrow A_{n+1}$.
\end{enumerate}
\end{lemma}
\begin{proof}
Induction on the cut-free proof of $\Gamma[\langle \Delta \rangle] \seqarrow A_{n+1}$.

First, suppose that the displayed occurrences of $\langle$ and $\rangle$ in $\Gamma[\langle \Delta \rangle] \seqarrow A_{n+1}$ are introduced at the last step of the proof.
There are two cases to consider.

\textit{Case 1.}
$\Gamma[\contexthole] = \contexthole$, $A_{n+1} = \Diamond A_{n+1}'$, and $\langle \Delta \rangle \seqarrow \Diamond A_{n+1}'$ is inferred from $\Delta \seqarrow A_{n+1}'$ by $({\seqarrow}\Diamond)$.
Let $B \in \Typ(\mathcal{B} \cup \{\Diamond \bm{1}\})$ be the interpolant for $\Delta \seqarrow A_{n+1}'$ with respect to the partition $(\Delta; \contexthole)$ obtained by the method of Theorems~\ref{thm:interpolation} and~\ref{thm:interpolation-unit-guarded}.
Then the interpolant for $\langle \Delta \rangle \seqarrow \Diamond A_{n+1}'$ with respect to the partition $(\langle \Delta \rangle; \contexthole)$ is $\Diamond B$.
By Theorem~\ref{thm:interpolation-unit-guarded}, condition (i) of the present theorem holds and $||\Diamond B|| \leq m$, which implies $||B|| \leq m-2$.

\textit{Case 2.}
$\Delta = \varbox C$ and $\Gamma[\langle \varbox C \rangle] \seqarrow A_{n+1}$ is inferred from $\Gamma[C] \seqarrow A_{n+1}$ by $(\varbox{\seqarrow})$.
Let $B \in \Typ(\mathcal{B} \cup \{\Diamond \bm{1}\})$ be the interpolant for $\Gamma[C] \seqarrow A_{n+1}$ with respect to the partition $(C; \Gamma[\contexthole])$ obtained by the method of Theorems~\ref{thm:interpolation} and~\ref{thm:interpolation-unit-guarded}.
Then the interpolant for $\Gamma[\langle \varbox C \rangle] \seqarrow A_{n+1}$ with respect to the partition $(\varbox C; \Gamma[\langle \contexthole \rangle])$ is $\varbox B$.
By Theorem~\ref{thm:interpolation-unit-guarded}, condition (ii) of the present theorem holds and $||\varbox B|| \leq m$, which implies $||B|| \leq m-2$.

Now suppose that the displayed occurrences of $\langle$ and $\rangle$ in $\Gamma[\langle \Delta \rangle] \seqarrow A_{n+1}$ are not introduced at the last step of the proof.
The last inference of the proof has one or two premises, one of which must be of the form $\Gamma'[\langle \Delta' \rangle] \seqarrow A_{n+1}'$, where either $\Delta'$ is identical to $\Delta$ or $\Gamma'[\contexthole] \seqarrow A_{n+1}'$ is identical to $\Gamma[\contexthole] \seqarrow A_{n+1}$.
If there is another premise, let that premise be $\Phi \seqarrow C$.
Let $A_1',\dots,A_k'$ be the yield of $\Gamma'[\langle \Delta' \rangle]$.
By the subformula property of cut-free proofs, we must have $A_i' \in \Typ(\mathcal{B} \cup \{ \Diamond \bm{1} \})$ and $||A_i'|| \leq m$ for each $i \in \{1,\dots,k,n+1\}$.
By the induction hypothesis, there is a type $B \in \Typ(\mathcal{B} \cup \{ \Diamond \bm{1} \})$ with $||B|| \leq m-2$ such that one of the following conditions holds:
\begin{enumerate}
\item[($\text{i}'$)]
$\vdash_{\mathbf{L}\kleene_{\bm{1}}\Diamond} \Delta' \seqarrow B$ and $\vdash_{\mathbf{L}\kleene_{\bm{1}}\Diamond} \Gamma'[\Diamond B] \seqarrow A_{n+1}'$.
\item[($\text{ii}'$)]
$\vdash_{\mathbf{L}\kleene_{\bm{1}}\Diamond} \Delta' \seqarrow \varbox B$ and $\vdash_{\mathbf{L}\kleene_{\bm{1}}\Diamond} \Gamma'[B] \seqarrow A_{n+1}'$.
\end{enumerate}
In each case, one of the conjuncts is identical to one of the conjuncts in (i) or (ii).
The other conjunct of (i) or (ii) is obtained from the corresponding conjunct of ($\text{i}'$) or ($\text{ii}'$) by applying the rule of inference used at the last step of the given proof of $\Gamma[\langle \Delta \rangle] \seqarrow A_{n+1}$ (using $\Phi \seqarrow C$ as the other premise if the rule is a two-premise rule).
\qed
\end{proof}

\begin{lemma}
\label{lem:cut-completeness-guarded}
Let\/ $\Gamma \seqarrow A_{n+1}$ be an\/ $\mathbf{L}\kleene_{\bm{1}}\Diamond$ sequent such that the yield of\/ $\Gamma$ is\/ $A_1 \dots A_n$ with\/ $A_i \in \Typ(\mathcal{B} \cup \{ \Diamond \bm{1} \})$ and\/ $||A_i|| \leq m$ for\/ $i=1,\dots,n+1$.
Then\/ $\vdash_{\mathbf{L}\kleene_{\bm{1}}\Diamond} \Gamma \seqarrow A_{n+1}$ if and only if\/ $\mathcal{T}'_{\mathcal{B},m} \vdash_{\mathrm{Cut}} \Gamma \seqarrow A_{n+1}$.
\end{lemma}
\begin{proof}
As before, the ``if'' direction is easy and the ``only if'' direction is by induction on the number of occurrences of brackets in $\Gamma$.
If $\Gamma = \Gamma'[\langle \rangle]$, then $\Gamma'[\langle \rangle] \seqarrow A_{n+1}$ is derivable from $\Gamma'[\Diamond \bm{1}] \seqarrow A_{n+1}$ and $\langle \rangle \seqarrow \Diamond \bm{1}$ by Cut.
Since by assumption $\vdash_{\mathbf{L}\kleene_{\bm{1}}\Diamond} \Gamma'[\langle \rangle] \seqarrow A_{n+1}$, we get $\vdash_{\mathbf{L}\kleene_{\bm{1}}\Diamond} \Gamma'[\Diamond \bm{1}] \seqarrow A_{n+1}$ using $(\bm{1}{\seqarrow})$ and $(\Diamond{\seqarrow})$.
By induction hypothesis, $\mathcal{T}'_{\mathcal{B},m} \vdash_{\mathrm{Cut}} \Gamma'[\Diamond \bm{1}] \seqarrow A_{n+1}$.
Since $\langle \rangle \seqarrow \Diamond \bm{1}$ is in $\mathcal{T}'_{\mathcal{B},m}$, it follows that $\mathcal{T}'_{\mathcal{B},m} \vdash_{\mathrm{Cut}} \Gamma'[\langle \rangle] \seqarrow A_{n+1}$.
The remaining cases are handled exactly as before.
\qed
\end{proof}

\begin{theorem}
\label{thm:L*Diamond-context-free}
Every language recognized by\/ $\mathbf{L}\kleene\Diamond$ is context-free.
\end{theorem}
\begin{proof}
Let $G = (\Sigma,I,D)$ be an $\mathbf{L}\kleene\Diamond$ grammar and define $\mathcal{B}$ and $m$ as in the proof of Theorem~\ref{thm:LDiamond-context-free}.
The definition of the context-free grammar $G' = (N,\Sigma,P,D)$ equivalent to $G$ is modified from the proof of Theorem~\ref{thm:LDiamond-context-free} as follows:
\begin{align*}
N & = \{\, A \in \Typ(\mathcal{B} \cup \{ \Diamond \bm{1} \}) \mid ||A|| \leq m \,\},\\
P & = 
\begin{aligned}[t]
& \{\, A_{n+1} \rightarrow A_1 \dots A_n \mid \text{$A_1 \dots A_n \seqarrow A_{n+1}$ is in $\mathcal{S}'_{\mathcal{B},m}$} \,\} \cup {}\\
& \{ \Diamond \bm{1} \rightarrow \epsilon \} \cup {}\\
& \{\, \Diamond A \rightarrow A \mid \text{$A \in \Typ(\mathcal{B} \cup \{ \Diamond \bm{1} \})$ and $||A|| \leq m-2$} \,\} \cup {}\\
& \{\, A \rightarrow \varbox A \mid \text{$A \in \Typ(\mathcal{B} \cup \{ \Diamond \bm{1})$ and $||A|| \leq m-2$} \,\} \cup {}\\
& \{\, A \rightarrow a \mid (a,A) \in I \,\}.
\end{aligned}
\end{align*}
Using Lemma~\ref{lem:cut-completeness-guarded}, we can prove that whenever $A_i \in \Typ(\mathcal{B} \cup \{ \Diamond \bm{1} \})$ and $||A_i|| \leq m$ for $i=1,\dots,n+1$, the following are equivalent:
\begin{enumerate}
\item[(i)]
$\vdash_{\mathbf{L}\kleene_{\bm{1}}\Diamond} \Gamma \seqarrow A_{n+1}$ for some $\Gamma$ whose yield is $A_1 \dots A_n$.
\item[(ii)]
$A_{n+1} \Rightarrow\kleene_{G'} A_1 \dots A_n$.
\item[(i$'$)]
$\mathcal{T}'_{\mathcal{B},m} \vdash_{\mathrm{Cut}} \Gamma \seqarrow A_{n+1}$ for some $\Gamma$ whose yield is $A_1 \dots A_n$.
\end{enumerate}
Since cut elimination holds of $\mathbf{L}\kleene_{\bm{1}}\Diamond$, when $\bm{1}$ does not occur in $\Gamma \seqarrow D$, we have $\vdash_{\mathbf{L}\kleene\Diamond} \Gamma \seqarrow D$ if and only if $\vdash_{\mathbf{L}\kleene_{\bm{1}}\Diamond} \Gamma \seqarrow D$.
This implies that $G$ and $G'$ are equivalent.
\qed
\end{proof}

\section{Conclusion}

We have shown that the calculi $\mathbf{L}\Diamond$ and $\mathbf{L}\kleene\Diamond$ both recognize only context-free languages.
The necessary ingredients of the proof were all available from Pentus's and J\"ager's work \citep{Pentus:1993,Pentus:1997,Jaeger:2003}.
Clearly, the same proof works for the multimodal generalizations of the calculi, $\mathbf{L}\Diamond_{\mathrm{m}}$ and $\mathbf{L}\kleene\Diamond_{\mathrm{m}}$.
The question of the recognizing power of the calculi with the unit, $\mathbf{L}\kleene_{\bm{1}}\Diamond$ and $\mathbf{L}\kleene_{\bm{1}}\Diamond_{\mathrm{m}}$, is left open.

\bibliographystyle{spbasic}
\bibliography{ref}

\end{document}